\documentclass[30pt,leqno]{article}
\usepackage{amsmath,amssymb}
\usepackage{authblk}
\usepackage{mathrsfs}
\usepackage{hyperref}
\usepackage{xcolor}

\newtheorem{theorem}{Theorem}[section]

\newtheorem{claim}[theorem]{Claim}
\newtheorem{corollary}[theorem]{Corollary}
\newtheorem{lemma}[theorem]{Lemma}
\newtheorem{definition}[theorem]{Definition}

\newtheorem{problem}[theorem]{Problem}

\newenvironment{proof}{\noindent\mbox{\bf Proof.}}{\mbox{$\dashv$}\bigskip}

\def\S{\mathcal{S}}
\def\halt{{\sf H}}
\def\PA{{\sf PA}}
\def\HA{{\sf HA}}

\def\tr{{\sf Tr}}
\def\II{{\cal I}}
\def\ax{{\sf Axiom}}

\def\lr{\leftrightarrow}
\def\Ra{\Rightarrow}
\def\La{\Leftarrow}
\def\LR{\Leftrightarrow}
\def\all{\forall}
\def\ex{\exists}

\def\M{{\mathfrak M}}
\def\MM{{\cal M}}
\def\K{{\bf K}}
\def\N{{\cal N}}
\def\KM{\K(\M,\T)}

\def\r{{\bf q}^{PR}}

\def\x{{\vec{x}}}
\def\y{{\vec{ y}}}

\def\L{\mathcal{L}}
\def\r{\:{\bf r}\:}
\def\jl{j_1}
\def\jr{j_2}
\def\T{{\bf T_\M}}
\def\ect{\mathsf{ECT_0}}
\def\th{{\bf Th}}
\def\un{\underline}
\def\n{\bar{n}}
\def\TT{{\sf T}}
\def\I{{\bf I}}
\def\MP{{\sf MP}}
\def\diag{\sf Diag}
\def\H{\diag(\M)}
\def\D{{\sf Dom}}
\def\prof{{\sf Proof}}
\def\pr{{\sf Pr}}
\def\con{{\sf Con}}
\def\ep{{\sf EP}}
\def\ext{{\sf EXT}}
\def\rfn{{\sf RFN}}
\def\mp{{\sf MP}}
\def\C{{\bf C}}
\def\PRA{{\sf PRA}}
\def\U{{\cal U}}

\DeclareFontFamily{U}{MnSymbolC}{}
\DeclareFontShape{U}{MnSymbolC}{m}{n}{
    <-6>  MnSymbolC5
    <6-7>  MnSymbolC6
    <7-8>  MnSymbolC7
    <8-9>  MnSymbolC8
    <9-10> MnSymbolC9
    <10-12> MnSymbolC10
    <12->   MnSymbolC12}{}
\DeclareSymbolFont{MnSyC}{U}{MnSymbolC}{m}{n}
\DeclareMathSymbol{\dotminus}{\mathbin}{MnSyC}{24}
\begin{document}
\title{Not all Kripke models of $\HA$ are locally $\PA$}
\author[1,2]{Erfan Khaniki\thanks{e.khaniki@gmail.com}}

\affil[1]{Faculty of Mathematics and Physics\\
Charles University}
\affil[2]{Institute of Mathematics\\
Czech Academy of Sciences}

\maketitle
{\centering\footnotesize To Mohammad Ardeshir\par}
\begin{abstract}
    Let $\K$ be an arbitrary Kripke model of Heyting Arithmetic, $\HA$. For every node $k$ in $\K$, we can view the classical structure of $k$, $\M_k$ as a model of some classical theory of arithmetic. Let $\TT$ be a classical theory in the language of arithmetic. We say $\K$ is locally $\TT$, iff for every $k$ in $\K$, $\M_k\models\TT$. One of the most important problems in the model theory of $\HA$ is the following question: {\it Is every Kripke model of $\HA$ locally $\PA$?}
    We answer this question negatively. We introduce two new Kripke model constructions to this end. The first construction actually characterizes the arithmetical structures that can be the root of a Kripke model $\K\Vdash\HA+\ect$ ($\ect$ stands for Extended Church Thesis). The characterization says that for every arithmetical structure $\M$, there exists a rooted Kripke model $\K\Vdash\HA+\ect$ with the root $r$ such that $\M_r=\M$ iff $\M\models\th_{\Pi_2}(\PA)$. One of the consequences of this characterization is that there is a rooted Kripke model $\K\Vdash\HA+\ect$ with the root $r$ such that $\M_r\not\models\I\Delta_1$ and hence $\K$ is not even locally $\I\Delta_1$. The second Kripke model construction is an implicit way of doing the first construction which works for any reasonable consistent intuitionistic arithmetical theory $\TT$ with a recursively enumerable set of axioms that has the existence property. We get a sufficient condition from this construction that describes when for an arithmetical structure $\M$, there exists a rooted Kripke model $\K\Vdash \TT$ with the root $r$ such that $\M_r=\M$. As applications of this sufficient condition, we construct two new Kripke models. The first one is a Kripke model $\K\Vdash \HA+\neg\theta+\mp$ ($\theta$ is an instance of $\ect$ and $\mp$ is Markov's principle) which is not locally $\I\Delta_1$. The second one is a Kripke model $\K\Vdash\HA$ such that $\K$ forces exactly the sentences that are provable from $\HA$, but it is not locally $\I\Delta_1$. Also, we will prove that every countable Kripke model of intuitionistic first-order logic can be transformed into another Kripke model with the full infinite binary tree as the Kripke frame such that both Kripke models force the same sentences. So with the previous result, there is a binary Kripke model $\K$ of $\HA$ such that $\K$ is not locally $\I\Delta_1$. 
\end{abstract}
\section{Introduction}
Heyting Arithmetic ($\HA$) is the intuitionistic counterpart of Peano Arithmetic ($\PA$). $\HA$ has the same non-logical axioms as $\PA$ with intuitionistic first-order logic as the underlying logic. This theory is one of the well-known and most studied theories of constructive mathematics, and it was investigated in many proof-theoretic and model-theoretic aspects in the literature (see \cite{tv} for more information). This paper aims to answer a question about the model theory of $\HA$. Let $\TT$ be a classical theory in the language of arithmetic. A Kripke model of $\HA$ is called locally $\TT$, iff for every node $k\in\K$, the classical structure $\M_k$ associated with $k$, is a model of $\TT$. One of the most important problems in the model theory of $\HA$ is the following question:
\begin{problem}\label{p1}
        Is every Kripke model of $\HA$ locally $\PA$?
\end{problem}
This problem was first asked and investigated in the seminal paper \cite{van} by van Dalen et al. in 1986. They proved that every finite Kripke model of $\HA$ is locally $\PA$. Furthermore, they proved that a Kripke model of $\HA$ with the Kripke frame $(\omega, \leq)$ as the underlying frame has infinitely many locally $\PA$ nodes. This work initiated a research line into Problem \ref{p1} and also about the following general question:
\begin{problem}\label{p2}
        For a Kripke model $\K$ of the theory $\TT$ in a language $\sigma$, and a node $k\in\K$, what is the relationship between the sentences forced in $k$ and the sentences satisfied in $\M_k$?
\end{problem}
There are several works that deal with these problems. We will review those works in the following paragraphs. Wehmeier in \cite{weh},  investigated Problem \ref{p1} and extended the results of \cite{van} to a larger class of frames. In particular, he proved that every Kripke model of $\HA$ with $(\omega, \leq)$ as the Kripke frame is indeed locally $\PA$. Moniri, in \cite{mon}, considered these problems and proved that every once-branching Kripke model of $\HA+\MP$ (Markov's principle) is locally $\PA$. Ardeshir and Hesaam in \cite{ah} generalized the results of \cite{weh} to rooted narrow tree Kripke models of $\HA$. Recently, Mojtahedi in \cite{moj} considered Problem \ref{p2} and answered this problem in the case of finite depth Kripke models. As an application, he generalized the result of \cite{ah} to rooted semi-narrow tree Kripke models of $\HA$.

Regarding Problem \ref{p1}, the strongest positive result about the strength of induction axioms that are true in a node of a Kripke model of $\HA$ was proved by Marković in \cite{mar}. He proved that every node of a Kripke model of $\HA$ satisfies induction for formulas that are provably $\Delta_1$ in $\PA$. Also, from $\Pi_2$ conservativity of $\PA$ over $\HA$ (see \cite{fr}), we know that every Kripke model of $\HA$ is locally $\th_{\Pi_2}(\PA)$.

Buss studied another question related to these problems in \cite{bu}. For every language $\sigma$ and every classical theory $\TT$ in it, he characterized the sentences that are true in every locally $\TT$ Kripke model. As a result, he proved that $\HA$ is complete with respect to the locally $\PA$ Kripke models. In a similar direction, Ardeshir et al. in \cite{ars} presented a set of axiom systems for the class of end-extension Kripke models. As an application, they proved that $\HA$ is strongly complete for its class of end-extension Kripke models. For the case of fragments of $\HA$, Problem \ref{p1} was investigated and answered negatively in \cite{po06}.

To best of our knowledge, the above theorems are all results relevant to Problem \ref{p1} in the literature. There are some other papers such as \cite{abz1,abz2} that investigated Problem \ref{p2} in general and partially answered this question. 

In this paper, we will present two new model construction to answer Problems \ref{p1} and \ref{p2}. The main technical theorem of the first construction says that the theory $\HA+\ect+\H$ for every $\M\models\th_{\Pi_2}(\PA)$ has the existence and disjunction properties (Theorem \ref{t6}). This theorem provides the right tool for constructing rooted Kripke models of $\HA$ with control over the structure of the root (Theorem \ref{t8}). This construction theorem moreover characterizes the necessary and sufficient conditions for an arithmetical structure $\M$ to be the root of a Kripke model of $\HA+\ect$ (Corollary \ref{char}). Using this characterization we will construct a Kripke model of $\HA+\ect$ that is not even locally $\I\Delta_1$. This answers Problem \ref{p1} negatively. Moreover, this is optimal, because it is well-known that every node of a Kripke model of $\HA$ satisfies induction for formulas that are provably $\Delta_1$ in $\PA$. The second construction is an implicit way of doing the first construction and it works for any reasonable consistent intuitionistic arithmetical theory with a recursively enumerable set of axioms that has the existence property (Theorem \ref{t14}). This construction gives us a sufficient condition for an arithmetical structure $\M$ to be the root of a Kripke model of $\TT$. As applications of this sufficient condition, we will construct two new Kripke models. The first one is a Kripke model of $\HA+\neg\theta+\mp$ where $\theta$ is an instance of $\ect$ and $\mp$ is Markov's principle that is not locally $\I\Delta_1$ (Corollary \ref{finF}). The second one is a Kripke model of $\HA$ that forces exactly all sentences that are provable in $\HA$, but it is not locally $\I\Delta_1$ (Corollary \ref{forceHA}). 

The second construction is general and also works for $\HA+\ect$, but some Kripke models can be constructed for $\HA+\ect$ with the first construction, but not possible with the second one. We will discuss this matter in more detail at the end of Section 3. The new model constructions imply the existence of a large class of Kripke models of reasonable intuitionistic arithmetical theories including $\HA$, which cannot be constructed by previous methods, so we think that these model constructions are interesting in their own rights.

We will also prove that every countable Kripke model of intuitionistic first-order logic can be transformed into another Kripke model with the full infinite binary tree as the Kripke frame (Lemma \ref{bin}). Using this result, we will prove that there exists a Kripke model of $\HA$ with the full infinite binary tree as the Kripke frame that is not locally $\I\Delta_1$ (Corollary \ref{bin2}).

\section{Preliminaries}
\subsection{Arithmetical Theories}
Let $\L$ be the language of Primitive Recursive Arithmetic in which it has a function symbol for every primitive recursive function. $\HA$ is the intuitionistic theory with the following non-logical axioms:
\begin{enumerate}
\item Axioms of Robinson Arithmetic $\sf Q$.
\item Axioms defining the primitive recursive functions.
\item For each formula $\phi(x,\vec{y})\in\L$, the axiom $\all \vec{y}\:{\bf I}_\phi$ in which $${\bf I}_\phi:=\phi(\bar{0})\land\all x(\phi(x)\to \phi(Sx))\to\all x\phi(x).$$
\end{enumerate}
$\PA$ is the classical theory that has the same non-logical axioms as $\HA$. $i\PRA$ (intuitionistic Primitive Recursive Arithmetic) has axioms of $\sf Q$ and induction for every atomic formula of $\L$. The underlying logic of $i\PRA$ is intuitionistic logic. $\PRA$ is the classical counter part of $i\PRA$. $\TT\vdash_c\phi$ means that there exists a proof of $\phi$ from axioms of $\TT$ using first-order classical logic Hilbert system. $\vdash_i$ denotes the same thing for intuitionistic proofs. An important set of intuitionistic arithmetical theories for the purpose of this paper is defined in the following definition.
\begin{definition}
$\II$ is the set of all intuitionistic arithmetical theories $\TT$ in $\L$ such that:
\begin{enumerate}
    \item $\TT$ is consistent.
    \item $i\PRA\subseteq \TT$.
    \item The set of axioms of $\TT$ is recursively enumerable.
\end{enumerate}
\end{definition}
Note that with the power of primitive recursive functions we can define finite sequences of numbers, so we can code finite objects such as formulas, proofs, and etc. as numbers. This is a standard technique and it is called G\"odel numbering (see \cite{sm1}). With the help of this coding we can talk about proofs of theories in arithmetical theories (see \cite{sm1}). For every $\L$ sentence $\phi$, $\ulcorner\phi\urcorner$ denotes the number associated with $\phi$. If $\phi(x)$ is an $\L$ formula, then $\ulcorner\phi(\dot{c})\urcorner$ denotes the number associated with $\psi(x)$ when we substitute the numeral with value $c$ for $x$. Suppose $\TT\in\II$. Let $\ax(x,y)$ be the primitive recursive function such that for every $\L$ sentence $\phi$, $\phi$ is a $\TT$-axiom iff $\ex x\ax(x,\ulcorner\phi\urcorner)$ is true. Then it is possible to define the provability predicate of $\TT$, $\prof_\TT(x,y)$ as a primitive recursive predicate as follows. Let $\left<.\right>$ be a natural primitive recursive coding function. Then $\prof_\TT(x,y)$ is true iff there exist two sequences $\L$ sentences $\{\phi_i\}_{i\leq n}$ and numbers $\{w_i\}_{i\leq n}$ for some $n$ such that:
\begin{enumerate}
    \item $x=\left<\left <w_1,\ulcorner\phi_1\urcorner\right>,...,\left<w_n,\ulcorner\phi_n\urcorner\right>\right>$.
    \item For every $i\leq n$:
    \begin{enumerate}
        \item If $w_i>0$, then $\ax(w_i-1,\ulcorner\phi_i\urcorner)$ is true.
        \item If $w_i=0$, then $\phi_i$ can be derived from $\{\phi_j\}_{j<i}$ by one of the rules of Natural deduction system for intuitionistic first-order logic.
    \end{enumerate}
    \item $y=\ulcorner\phi_n\urcorner$.
\end{enumerate}
 The $\Sigma_1$ formula $\pr_\TT(y)$ is the abbreviation for $\ex x\prof(x,y)$. So consistency of $\TT$, $\con(\TT)$, is $\neg\pr_\TT(\ulcorner\bot\urcorner)$. The following theorem states the useful facts about $\pr_\TT$.

\begin{theorem}
For every $\TT\in\II$ the following statements are true:
\begin{enumerate}\label{pr}
    \item For every $\L$ sentence $\phi$, if $\TT\vdash_i\phi$, then $\PRA\vdash_c\pr_\TT(\ulcorner\phi\urcorner)$.
    \item $\PRA\vdash_c \all x,y(\pr_\TT(x)\land \pr_\TT(x\to y)\to\pr_\TT(y))$.
    \item $\PRA\vdash_c \all x,y(\pr_\TT(x)\land\pr_\TT(y)\to\pr_\TT(x\land y))$.
    \item For every $\L$ formula $\phi(x)$ with $x$ as the only free variable, $\PRA\vdash_c\pr_\TT(\ulcorner\all x\phi(x)\urcorner)\to\all x\pr_\TT(\ulcorner\phi(\dot{x})\urcorner)$.
    \item For every $\Sigma_1$ formula $\phi(\vec{x})$, $\PRA\vdash_c \all \vec{x}(\phi(\vec{x})\to\pr_\TT(\ulcorner\phi(\dot{x_1},...,\dot{x_n})\urcorner))$.
\end{enumerate}
\end{theorem}
\begin{proof}
See \cite{sm1} for a detailed discussion of these statements.
\end{proof}
\subsection{Realizability}
For proving the first model construction theorem, we need some definitions and theorems about Kleene's realizability.  
\begin{definition}
Let $T(x,y,z)$ be the primitive recursive function called Kleene's T-predicate and $U(x)$ be the primitive recursive function called result-extracting function. Note that $$\HA\vdash_i \all x,y,z,z'( T(x,y,z)=0\land T(x,y,z')=0\to z=z').$$ We use $T(x,y,z)$ instead of $T(x,y,z)=0$ for simplicity. For more information, see section 7 of the third chapter of \cite{tv}.
\end{definition}
Let $\jl(x)$ and $\jr(x)$ be the primitive recursive projections of the pairing function $j(x,y)=2^x\cdot(2y+1)\dotminus 1$. Kleene's realizability is defined as follows.
\begin{definition}
$x\r \phi$ ($x$ realizes $\phi$) is defined by induction on the complexity of $\phi$ where $x\not\in FV(\phi)$.
\begin{enumerate}
\item $x\r p:=p$ for atomic $p$,
\item $x\r(\psi\land \eta):=\jl(x)\r \psi \land \jr(x)\r \eta$,
\item $x\r(\psi\lor \eta):=(\jl(x)=0 \land \jr(x)\r \psi)\lor(\jl(x)\not=0\land \jr(x)\r \eta)$,
\item $x\r(\psi\to \eta):= \all y(y\r \psi\to \ex u(T(x,y,u)\land U(u)\r \eta)$, $u\not\in FV(\eta)$,
\item $x\r\ex y\psi(y):=\jr(x)\r \psi(\jl(x))$,
\item $x\r \all y\psi(y):=\all y\ex u(T(x,y,u)\land U(u)\r \psi(y))$, $u\not\in FV(\psi)$.
\end{enumerate}
\end{definition}
\begin{definition}
A formula $\phi\in\L$ is almost negative iff $\phi$ does not contain $\lor$, and $\ex$ only immediately in front of atomic formulas.
\end{definition}
\begin{definition}\label{d6}
The extended Church's thesis is the following schema, where $\phi$ is almost negative:
$$\ect:= \all\vec{v}\left(\all x(\phi(x,\vec{v})\to\ex y \psi(x,y,\vec{v}))\to\ex z \all x(\phi(x,\vec{v})\to \ex u(T(z,x,u) \land \psi(x,U(u),\vec{v})))\right).$$
\end{definition}
Next theorem explains the relationships between, $\HA$, $\ect$ and Kleene's realizability.
\begin{theorem}
\label{t1}
For every formula $\phi\in\L$:
\begin{enumerate}
\item $\HA+\ect\vdash_i \phi \lr \ex x (x\r \phi)$,
\item $\HA+\ect\vdash_i \phi \LR \HA\vdash_i \ex x(x\r \phi)$.
\end{enumerate}
\end{theorem}
\begin{proof}
See Theorem 4.10 in the fourth chapter of \cite{tv}.
\end{proof}

Another important properties of $\HA$ are the existence and disjunction properties. We will use notation $\bar{n}$ as the syntactic term corresponds to natural number $n$.
\begin{theorem}\label{dis}
The following statements are true:
\begin{enumerate}
    \item Disjunction property: For every sentences $\phi,\psi\in\L$, if $\HA\vdash_i \phi \lor \psi$, then $\HA\vdash_i \phi$ or $\HA\vdash_i \psi$,
    \item Existence property: For every sentence $\ex x\phi(x)\in\L$, if $\HA\vdash_i \ex x\phi(x)$, then there exists a natural number $n$ such that $\HA\vdash_i \phi(\bar{n})$.
\end{enumerate}
\end{theorem}
\begin{proof}
See Theorem 5.10 of the third chapter of \cite{tv}.
\end{proof}

Although $\HA$ is an intuitionistic theory, it can prove decidability of some restricted class of formulas. The next theorem explains this fact.
\begin{theorem}\label{dec}
For every quantifier free formula $\phi\in\L$, $\HA\vdash_i \phi \lor \neg \phi$.
\end{theorem}
\begin{proof}
See \cite{tv}.
\end{proof}
\subsection{Kripke models}
A Kripke model for a language $\sigma$ is a triple $\K=(K,\leq,\M)$ such that:
\begin{enumerate}
    \item $(K,\leq)$ is a nonempty partial order.
    \item For every $k\in K$, $\M_k\in\M$ is a classical structure in the language $\sigma({\M_k})=\sigma\cup\{\un{c}|c\in\M_k\}$.
    \item For every $k,k'\in K$, if $k\leq k'$, then $\sigma({\M_k})\subseteq \sigma({\M_{k'}})$ and also $\M_{k'}\models \diag^+(\M_k)$ ($\M_k$ is a sub-structure of $\M_{k'}$).
\end{enumerate}

For every Kripke model $\K$, there is a uniquely inductively defined relation $\Vdash \subseteq K \times \left(\bigcup_{k\in K}\sigma({\M_k})\right)$ that is called forcing.
\begin{definition}
For every $k\in K$, and every sentence $\phi\in\sigma({\M_k})$, the relation $k\Vdash \phi$ is defined by induction on complexity of $\phi$:
\begin{enumerate}
    \item $k\Vdash p$ iff $\M_k\models p$, for atomic $p$,
    \item $k\Vdash \psi \land \eta$ iff $k\Vdash \psi$ and $k\Vdash \eta$,
    \item $k\Vdash \psi \lor \eta$ iff $k\Vdash \psi$ or $k\Vdash \eta$,
    \item $k\Vdash \neg\psi$ iff for no $k'\geq k$, $k'\Vdash \psi$,
    \item $k\Vdash \psi\to\eta$ iff for every $k'\geq k$, if $k'\Vdash \psi$, then $k'\Vdash \eta$,
    \item $k\Vdash \ex x \psi(x)$ iff there exists $\un{c}\in\sigma_{\M_k}$ such that $k\Vdash \psi(\un{c})$,
    \item $k\Vdash \all x \psi(x)$ iff for every $k'\geq k$ and every $\un{c}\in\sigma({\M_{k'}})$, $k'\Vdash \psi(\un{c})$.
\end{enumerate}
\end{definition}

We use the notation $\K\Vdash \phi$ ($\phi\in\bigcap_{k\in K}\sigma({\M_k})$ is a sentence) as an abbreviation that for every $k\in K$, $k\Vdash \phi$ which simply means that the Kripke model $\K$ forces $\phi$. The important property of the forcing relation is its monotonicity. This means that for every $k'\geq k$ and every $\phi\in\sigma({\M_k})$, if $k\Vdash \phi$, then $k'\Vdash \phi$. Also, note that first-order intuitionistic logic is sound and is strongly complete with respect to the Kripke models. For more details see \cite{tv}.
\section{Kripke model constructions for intuitionistic arithmetical theories}
\subsection{The first model construction}
 We will explain the first model construction in this subsection. This construction will be presented in a sequence of lemmas and theorems.
\begin{lemma}
\label{l2}
For every quantifier-free formula $\phi\in\L$ there exists an atomic formula $p\in\L$ with the same free variables such that $\HA\vdash_i \phi\lr p$.
\end{lemma}
\begin{proof}
By induction on the complexity of $\phi$ and using Theorem \ref{dec}.
\end{proof}
\begin{lemma}
\label{l3}
Let $\left<.\right>$ and $(.)_x$ be a primitive recursive coding and decoding functions, then for every formula $Qx_1,...,x_n \phi(\x,\y)\in\L$ where $Q\in\{\all,\ex\}$ and $n>0$, $$\HA\vdash_i Qx_1,...,x_n \phi(\x,\y)\lr Qx \phi((x)_0,...,(x)_n,\y).$$
\end{lemma}
\begin{proof}
Straightforward by properties of the coding and decoding functions.
\end{proof}

We will use notation $\phi([x],\y)$ instead of $\phi((x)_0,...,(x)_n,\y)$ for simplicity.
\begin{theorem}
\label{t4}
For every $\Pi_2$ sentence $\phi:=\all \x\ex \y \psi(\x,\y)$, if $\HA+\ect\vdash_i \phi$, then $\PA\vdash_c \phi$.
\end{theorem}
\begin{proof}
Let $\phi$ be a $\Pi_2$ sentence and $\HA+\ect\vdash_i \phi$. By Lemmas \ref{l3} and \ref{l2} there exists an atomic formula $p(x,y)$ such that $\HA \vdash_i \phi\lr \all x\ex y p(x,y)$ and therefore $\HA+\ect \vdash_i \all x\ex yp(x,y)$. By Theorem \ref{t1} $\HA\vdash_i \ex n (n\r \all x\ex y p(x,y))$. Because $ \ex n (n\r \all x\ex y p(x,y))$ is a sentence, there exists a natural number $n$ such that $ \HA\vdash_i\bar{n}\r \all x\ex y p(x,y)$. Therefore by definition of the realizability:
\begin{enumerate}
\item $\Ra\HA\vdash_i \all x \ex u(T(\bar{n},x,u)\land U(u)\r \ex yp(x,y)),$
\item $\Ra\HA\vdash_i \all x \ex u(T(\bar{n},x,u)\land \jr(U(u))\r p(x,\jl(U(u)))),$
\item $\Ra\HA\vdash_i \all x \ex u(T(\bar{n},x,u)\land p(x,\jl(U(u)))),$
\item $\Ra\HA\vdash_i \all x \ex up(x,u),$
\end{enumerate}
hence $\PA\vdash_c \phi$.
\end{proof}

In the rest of the paper, for every $\L$ structure $\M$, $$\T:=\HA+\ect+\H.$$
\begin{theorem}
\label{t5}
If $\M\models \th_{\Pi_1}(\PA)$, then $\T$ is consistent.
\end{theorem}
\begin{proof}
Suppose $\T$ is inconsistent, so there exists a finite number of $\L(\M)$ sentences $\{\phi_i(\vec{\un{c}_i})\}_{i\leq n}\subseteq \H$ such that $\HA+\ect+\bigwedge_{i=1}^n\phi_i(\vec{\un{c}_i})\vdash_i\bot$, therefore $\HA+\ect\vdash_i\neg\bigwedge_{i=1}^n\phi_i(\vec{\un{c}_i})$. Because $\vec{\un{c}_i}$ are not used in the axioms of $\HA+\ect$, we have $\HA+\ect\vdash_i\all \vec{x_1},...,\vec{x_n}(\neg\bigwedge_{i=1}^n\phi_i(\vec{x_i}))$. Note that $\all \vec{x_1},...,\vec{x_n}(\neg\bigwedge_{i=1}^n\phi_i(\vec{x_i}))$ is a $\Pi_1$ sentence and therefore by Theorem \ref{t4}, $\PA\vdash_c \all \vec{x_1},...,\vec{x_n}(\neg\bigwedge_{i=1}^n\phi_i(\vec{x_i}))$. This implies that $\M\models \all \vec{x_1},...,\vec{x_n}(\neg\bigwedge_{i=1}^n\phi_i(\vec{x_i}))$ and especially $\M\models \neg\bigwedge_{i=1}^n\phi_i(\vec{\un{c}_i})$, but by definition of $\diag(\M)$ we know $\M\models \bigwedge_{i=1}^n\phi_i(\vec{\un{c}_i})$ and this leads to a contradiction, hence $\T$ is consistent.
\end{proof}

If an $\L$ structure $\M$ satisfies a strong enough theory of arithmetic, then $\T$ has actually the existence and disjunction properties.
\begin{theorem}
\label{t6}
(Existence and disjunction properties). Suppose $\M$ is a model of $\th_{\Pi_2}(\PA)$, then the following statements are true:
\begin{enumerate}
    \item For every $\L(\M)$ sentence $\ex z \phi(z)$ such that $\T\vdash_i \ex z \phi(z),$ there exists a constant symbol $\un{c}\in\L_\M$ such that $\T\vdash_i \phi(\un{c})$.
    \item For every $\L(\M)$ sentence $\phi\lor \psi$ such that $\T\vdash_i \phi\lor \psi$, $\T\vdash_i \phi$ or $\T\vdash_i \psi$.
\end{enumerate}
\end{theorem}
\begin{proof}
\begin{enumerate}
    \item Suppose $\phi(z)$ is $\psi(z,\vec{\un{d}})$ such that $\psi(z,\y)$ is an $\L$ formula. By assumption of the theorem there exists a finite number of $\L(\M)$ sentences $\{\phi_i(\vec{\un{c}_i})\}_{i\leq n}\subseteq \H$ such that $$\HA+\ect+\bigwedge_{i=1}^n\phi_i(\vec{\un{c}_i})\vdash_i\ex z \psi(z,\vec{\un{d}}),$$ so $\HA+\ect\vdash_i\bigwedge_{i=1}^n\phi_i(\vec{\un{c}_i})\to\ex z \psi(z,\vec{\un{d}})$. Because $\L(\M)$ constants that appear in $\bigwedge_{i=1}^n\phi_i(\vec{\un{c}_i})\to\ex z \psi(z,\vec{\un{d}})$ are not used in the axioms of $\HA+\ect$, therefore $$\HA+\ect\vdash_i \all\y,\x_1,...,\x_n(\bigwedge_{i=1}^n\phi_i(\x_i,\y)\to\ex z \psi(z,\y)).$$ Note that $\bigwedge_{i=1}^n\phi_i(\x_i,\y)$ is a quantifier free formula, hence by Lemma \ref{l2} there exists an atomic formula $p$ such that $\HA\vdash_i p(\x_1,...,x_n,\y)\lr\bigwedge_{i=1}^n\phi_i(\x_i,\y)$. Also note that by Theorem \ref{dec} $\HA\vdash_i p\lor\neg p$, hence $$\HA+\ect\vdash_i \all\y,\x_1,...,\x_n\ex z(p(\x_1,...,\x_n,\y)\to \psi(z,\y)).$$
By Lemma \ref{l3} $\HA+\ect\vdash_i\all x\ex z(p([x])\to \psi(z,[x]))$. Note that $\all x\ex z(p([x])\to \psi(z,[x]))$ is an $\L$ sentence and therefore by Theorems \ref{t1} and \ref{dis} there exists a natural number $n$ such that $$\HA\vdash_i \n\r\all x\ex z(p([x])\to \psi(z,[x])).$$ By definition of realizability we get $$\HA\vdash_i \all x\ex u(T(\n,x,u)\land U(u)\r\ex z(p([x])\to \psi(z,[x]))).$$ Note that $\HA\vdash_i \all x \ex u T(\n,x,u)$, hence $\PA\vdash_c \all x \ex u T(\n,x,u)$ and therefore $\M\models \all x \ex u T(\n,x,u)$. Let $\M\models e=\left<\vec{c_1},...,\vec{c_n},\vec{d}\right>$ and $\M\models T(\n,e,f)\land U(f)=g$ for some $e,f,g\in\M$. This implies $T(\n,\un{e},\un{f}),U(\un{f})=\un{g}\in\H$ and therefore we get $$\T\vdash_i T(\n,\un{e},\un{f})\land \un{g}\r\ex z(p([\un{e}])\to \psi(z,[\un{e}])).$$ By applying realizability definition we get $\T\vdash_i \jr(\un{g})\r(p([\un{e}])\to \psi(\jl(\un{g}),[\un{e}]))$. Note that by Theorem \ref{t1}, $$\HA+\ect\vdash_i v\r(p[x]\to \psi(w,[x]))\to (p([x])\to \psi(w,[x])),$$ so $$\T\vdash_i p([\un{e}])\to \psi(\jl(\un{g}),[\un{e}]).$$ Because $p([\un{e}])\in\H$, we get $\T\vdash_i \psi(\jl(\un{g}),[\un{e}])$ and this implies $\T\vdash_i \psi(\un{c},[\un{e}])$ for some $\un{c}\in\L(\M)$ such that $\M\models \jl(g)=c$.
\item Suppose $\T$ proves $\phi\lor \psi$, therefore $\T\vdash_i \ex x((x=0\to \phi)\land(x\not=0\to \psi))$. By the previous part there exists a constant symbol $\un{c}\in\L(\M)$ such that $\T\vdash_i (\un{c}=0\to \phi)\land(\un{c}\not=0\to \psi)$. Note that $\un{c}=0$ is an atomic formula, hence $\un{c}=0\in\H$ or $\un{c}\not=0\in\H$ and this implies  $\T\vdash_i \phi$ or $\T\vdash_i \psi$.
\end{enumerate}
\end{proof}
\begin{definition}
Let $\M$ be an $\L$ structure and $\TT$ be an intuitionistic theory in the language $\L(\M)$. For every $\L(\M)$ sentence $\phi$ such that $\TT\nvdash_i \phi$, fix a Kripke model $\K_\phi\Vdash \TT$ such that $\K_\phi\nVdash \phi$.
\end{definition}
The following definition is based on Smory\'nski collection operation in \cite{sm}.
\begin{definition}
Let $\M$ be an $\L$ structure and $\TT$ be an intuitionistic theory in the language $\L(\M)$. Define $$\S(\M,\TT)=\{\phi\in\L(\M)|\TT\nvdash_i \phi, \phi \text{ is a sentence}\}.$$ Define the universal model $\K(\M,\TT)$ as follows. Take the disjoint union $\{\K_\phi\}_{\phi\in \S(\M,\TT)}$ and then add a new root $r$ with domain $\M_r=\M$.
\end{definition}
\begin{theorem}
\label{t8}
If $\M$ is a model of $\th_{\Pi_2}(\PA)$, then $\K(\M,\T)$ is a well-defined Kripke model and for every $\L(\M)$ sentence $\phi$, $$\K(\M,\T)\Vdash \phi\LR\T\vdash_i \phi.$$ 
\end{theorem}
\begin{proof}
First note that by Theorem \ref{t5} $\T\nvdash \bot$, hence $\S(\M,\T)$ is not empty and therefore $\K(\M,\T)$ has other nodes except $r$. To make sure that $\KM$ is well-defined, we should check the three conditions in the definition of Kripke models. It is easy to see that the first two conditions hold for $\KM$. For the third condition, we need to show that for every node $k\neq r$, $\L(\M_r)\subseteq\L(\M_k)$ and $\M_k\models\diag^+(\M_r)$. By definition of $\KM$, $\L(\M_r)\subseteq\L(\M_k)$ holds. For the condition $\M_k\models\diag^+(\M_r)$, note that $\T\vdash_i \diag(\M)$ which implies $\M_k\models\diag(\M_r)$.
\begin{itemize}
\item [($\Ra$).] Let $\K(\M,\T)\Vdash \phi$. If $\T\nvdash_i \phi$, then $\K_\phi$ exists and $\K_\phi\subseteq \K(\M,\T)$. By the assumption we get $\K_\phi\Vdash \phi$, but this leads to a contradiction by definition of $\K_\phi$, hence $\T\vdash_i \phi$. 
\item [($\La$).] We prove this part by induction on the complexity of $\phi$:
\begin{enumerate}
\item $\phi=p$: Note that if $\T\vdash_i p$, then $p\in\H$. Because if $p\not\in\H$, then $\neg p\in \H$, hence $\T\vdash_i \bot$ which leads to a contradiction by Theorem \ref{t5}. Therefore $p\in\H$ and by the fact that $\M\models p$ we get $\K(\M,\T)\Vdash p$.
\item $\phi=\psi\land \eta$: By the assumption we get $\T\vdash_i \psi$ and $\T\vdash_i \eta$, therefore by the induction hypothesis $\K(\M,\T)\Vdash \psi$ and $\KM\Vdash \eta$, hence $\KM\Vdash \psi\land \eta$.
\item $\phi=\psi\lor \eta$: By Theorem \ref{t6} $\T\vdash_i \psi$ or $\T\vdash_i \eta$, therefore by the induction hypothesis $\KM\Vdash \psi$ or $\KM\Vdash \eta$, hence $\KM\Vdash \psi\lor \eta$.
\item $\phi=\psi\to \eta$: By the assumption for every $\theta\in \S(\M,\T)$, $\K_\theta\Vdash \psi\to \eta$, so for proving $\KM\Vdash \psi\to \eta$ we only need to show that if $r\Vdash \psi$, then $r\Vdash \eta$. Let $r\Vdash \psi$, therefore we have $\KM\Vdash \psi$, hence by the previous part, $\T\vdash_i \psi$. Note that By the assumption $\T\vdash_i \psi\to \eta$, hence $\T\vdash_i \eta$ and therefore by the induction hypothesis $\KM\Vdash \eta$ which implies $r\Vdash \eta$.
\item $\phi=\ex x \psi(x)$: By Theorem \ref{t6} there exists a constant symbol $\un{c}\in\L(\M)$ such that $\T\vdash_i \psi(\un{c})$, therefore by the induction hypothesis $\KM\Vdash \psi(\un{c})$, hence $\KM\Vdash \ex x\psi(x)$.
\item $\phi=\all x\psi(x)$: By the assumption for every $\theta\in \S(\M,\T)$, $\K_\theta\Vdash \all x\psi(x)$, so for proving $\KM\Vdash \all x\psi(x)$ we only need to show that for every $c\in \M$, $r\Vdash \psi(\un{c})$. Let $c\in \M$. By the assumption $\T\vdash_i \all x \psi(x)$, therefore $\T\vdash_i \psi(\un{c})$, hence by induction hypothesis $\KM\Vdash \psi(\un{c})$. This implies that $r\Vdash \psi(\un{c})$. Note that $\un{c}$ is interpreted by $c\in \M$, hence $k\Vdash \psi(c)$.
\end{enumerate}
\end{itemize}
\end{proof}

From the last theorem, we can get the characterization of the structure of the roots of Kripke models of $\HA+\ect$.
\begin{corollary}\label{char}
For every $\L$ structure $\M$, there exists a rooted Kripke model $\K\Vdash \HA+\ect$ with the root $r$ such that $\M_r=\M$ iff $\M\models\th_{\Pi_2}(\PA)$.
\end{corollary}
\begin{proof}
As we mentioned before, it is known that every Kripke model of $\HA$ is locally $\th_{\Pi_2}(\PA)$ which proves the left to the right direction.

For the case of right to left direction, note that if $\M\models\th_{\Pi_2}(\PA)$, then by Theorem \ref{t8} $\K(\M,\T)\Vdash \HA+\ect$ and moreover the classical structure attached to the root is $\M$.
\end{proof}

Now we have the right tool for constructing a counter example for Problem \ref{p1}. In general we can get a lot of new models for every $\M\models\th_{\Pi_2}(\PA)$. For our purpose, it is sufficient to know that $\th_{\Pi_2}(\PA)\nvdash_c \PA$ to get the result. The next two theorems established the stronger fact which says $\th_{\Pi_2}(\PA)\nvdash_c \I{\Delta_1}$. $\I\Delta_1$ is a classical theory in the language $\L$ with the following non-logical axioms:  
\begin{enumerate}
\item Axioms of Robinson Arithmetic $\sf Q$.
\item Axioms defining the primitive recursive functions.
\item $\Delta_1$ induction:
$$\all \vec{y}\left[\all x(\phi(x,\vec{y})\lr \neg\psi(x,\vec{y}))\to \I_\phi\right]$$
for every $\Sigma_1$ formulas $\phi,\psi\in\L$
\end{enumerate}

For stating the theorems we need also another arithmetical theory that is called ${\bf B}\Sigma_1$ with the following non-logical axioms:
\begin{enumerate}
\item Axioms of Robinson Arithmetic $\sf Q$.
\item Axioms defining the primitive recursive functions.
\item Induction for quantifier free formulas.
\item Bounded $\Sigma_1$ collection:
$$\forall \vec{y},x\left[\forall z(z<x\to\exists w\phi(z,w,\vec{y}))\to\exists r\forall z(z<x\to\exists w(w<r\land\phi(z,w,\vec{y}))\right]$$
for every $\Sigma_1$ formulas $\phi,\psi\in\L$
\end{enumerate}
It is worth mentioning that these theories usually are defined over the language of Peano Arithmetic, and not over the language of Primitive Recursive Arithmetic, hence our definitions of $\I{\Delta_1}$ and ${\bf B}\Sigma_1$ are stronger than the usual definition, but for our use this does not cause a problem. Now we know the definitions, we will state the theorems.
\begin{theorem}\label{eq}
$\I\Delta_1\:_c\dashv\vdash_c {\bf B}\Sigma_1$.
\end{theorem}
\begin{proof}
As we explained before, this version of these theories are stronger that the original ones. Therefore by the result of \cite{sl} these two theories are the same.
\end{proof}
\begin{theorem}
\label{t9}
There exists a model $\M\models \th_{\Pi_2}(\mathbb{N})$ such that $\M\not\models \I\Delta_1$.
\end{theorem}
\begin{proof}
By the result of \cite{r} there exists a model $\M\models\th_{\Pi_2}(\mathbb{N})$ such that $\M\not\models{\bf B}\Sigma_1$, hence by Theorem \ref{eq} $\M\not\models\I\Delta_1$ too.
\end{proof}

\begin{corollary}\label{fin}
There exists a rooted Kripke model of $\HA+\ect$ which is not locally $\I\Delta_1$.
\end{corollary}
\begin{proof}
By Theorem \ref{t9} there exists a model $\M\models \th_{\Pi_2}(\mathbb{N})$ such that $\M\not\models \I\Delta_1$. Note that by Theorem \ref{t8}, $\KM\Vdash \HA+\ect$, and also $\KM$ is not locally $\I\Delta_1$.
\end{proof}

$\ect$ is a very powerful non-classical axiom schema, so a natural question is that: {\it Is it the case that for every Kripke model $\K\Vdash \HA+\ect$ and every node $k$ in $\K$, $\M_k\not\models\PA$ ?} This question has a negative answer, because $\K(\mathbb{N},{\bf T}_\mathbb{N})\Vdash \HA+\ect$, but $\M_r\models\PA$.
\subsection{The second model construction}
In this subsection, we will explain the generalized construction which works for any reasonable intuitionistic arithmetical theory. We will also mention an application of it at the end of this subsection.

For every $\TT\in\II$, the existence property of $\TT$ is the following $\Pi_2$ sentence:
 $$\ep(\TT):=\all x(x=\ulcorner\ex y\phi(y)\urcorner\text{ for some formula }\phi(y) \land x\text{ is a sentence}\land \pr_\TT(x)\to\ex y\pr_\TT(\ulcorner\phi(\dot{y})\urcorner)).$$
For an $\L$ structure $\M$ and a theory $\TT\in\II$, let extension of $\TT$ with respect to $\M$ be the following theory:
$$\ext(\M,\TT):=\{\phi\in\L(\M)|\phi\text{ is a sentence}, \M\models\pr_\TT(\ulcorner\phi\urcorner)\}.$$ 
The following lemma states that $\ext(\M,\TT)$ is closed under finite conjunctions.
\begin{lemma}
Let $\M\models \PRA$ and $\TT\in\II$. Then for every $\L(\M)$ sentences $\phi$ and $\psi$, if $\phi,\psi\in\ext(\M,\TT)$, then $\phi\land\psi\in\ext(\M,\TT)$.
\end{lemma}
\begin{proof}
If $\phi,\psi\in\ext(\M,\TT)$, then $\M\models \pr_\TT(\ulcorner\phi\urcorner)\land\pr_\TT(\ulcorner\psi\urcorner)$, so by Theorem \ref{pr} $\M\models \pr_\TT(\ulcorner \phi\land\psi\urcorner)$. Hence $\phi\land\psi\in\ext(\M,\TT)$.
\end{proof}

Define 
$$\C_{\M,\TT}:=\TT+\ext(\M,\TT).$$
The crucial property of $\C_{\M,\TT}$ is the following lemma.
\begin{lemma}\label{l12}
Suppose $\M\models \PRA$. Then for every $\TT\in\II$ and every $\L(\M)$ sentence $\psi$, if $\C_{\M,\TT}\vdash_i \psi$, then $\M\models \pr_\TT(\ulcorner\psi\urcorner)$.
\end{lemma}
\begin{proof}
Let $\psi(\vec{\un{d}})$ be an $\L(\M)$ sentence such that $\C_{\M,\TT}\vdash_i \psi(\vec{\un{d}})$. So there exists a finite number of $\L(\M)$ sentence $\{\phi_i(\vec{\un{c}_i})\}_{i\leq n}\subseteq \ext(\M,\TT)$ such that $$\TT\vdash_i \bigwedge_{i=1}^n\phi_i(\vec{\un{c}_i})\to\psi(\vec{\un{d}}).$$ 
Because $\L(\M)$ constants that appear in $\bigwedge_{i=1}^n\phi_i(\vec{\un{c}_i})\to \psi(\vec{\un{d}})$ are not used in the axioms of $\TT$, therefore $$\TT\vdash_i \all\y,\x_1,...,\x_n(\bigwedge_{i=1}^n\phi_i(\x_i,\y)\to \psi(\y)).$$
So by Theorem \ref{pr}
$$\M\models \pr_\TT(\ulcorner\all\y,\x_1,...,\x_n(\bigwedge_{i=1}^n\phi_i(\x_i,\y)\to \psi(\y))\urcorner).$$
Hence again by Theorem \ref{pr}
$$\M\models \pr_\TT(\ulcorner\bigwedge_{i=1}^n\phi_i(\vec{\dot{c}}_i)\to \psi(\vec{\dot{d}})\urcorner).$$
On the other hand by Lemma \ref{l12} $\ext(\M,\TT)$ is closed under finite conjunctions, so $\bigwedge_{i=1}^n\phi_i(\vec{\un{c}_i})\in\ext(\M,\TT)$ which means $\M\models\pr_\TT(\ulcorner\bigwedge_{i=1}^n\phi_i(\vec{\dot{c}}_i)\urcorner)$. So by Theorem \ref{pr} $\M\models\pr_\TT(\ulcorner\psi(\vec{\dot{d}})\urcorner)$.
\end{proof}
\begin{theorem}\label{t12}
For every $\TT\in\II$ and every $\M\models \PRA+\ep(\TT)+\con(\TT)$, the following statements are true:
\begin{enumerate}
    \item $\C_{\M,\TT}$ is consistent.
    \item $\C_{\M,\TT}$ has the existence and disjunction properties.
\end{enumerate}
\end{theorem}
\begin{proof}
$ $

\begin{enumerate}
    \item Suppose $\C_{\M,\TT}\vdash_i \bot$. Then by Lemma \ref{l12} $\M\models \pr_\TT(\ulcorner\bot\urcorner)$, but this is not possible because we assumed $\M\models\con(\TT)$, hence $\C_{\M,\TT}$ is consistent.
    \item We will prove the existence property of $\C_{\M,\TT}$. The disjunction property will follow from it by the same argument as in the proof of Theorem \ref{t6}. Let $\psi(x)$ be a formula in $\L(\M)$ with only $x$ as the free variable. Suppose $\C_{\M,\TT}\vdash_i \exists x \psi(x)$. Then by Lemma \ref{l12} $\M\models \pr_\TT(\ulcorner \ex x \psi(x)\urcorner)$. Note that $\M\models \ep(\TT)$, hence $\M\models \ex x \pr_\TT(\ulcorner\psi(\dot{x})\urcorner)$. This means there exists a $c\in\M$ such that $\M\models \pr_\TT(\ulcorner\psi(\dot{c})\urcorner)$. This implies $\psi(\un{c})\in\ext(\M,\TT)$, so $\C_{\M,\TT}\vdash_i \psi(\un{c})$.
\end{enumerate}
\end{proof}

This is the generalized version of the Theorem \ref{t8} which gives us the sufficient condition.
\begin{theorem}\label{t14}
Let $\TT\in\II$ and $\M\models\PRA+\ep(\TT)+\con(\TT)$. Then $\K(\M,\C_{\M,\TT})$ is a well-defined Kripke model and for every $\L(\M)$ sentence $\phi$, $$\K(\M,\C_{\M,\TT})\Vdash \phi \Leftrightarrow \C_{\M,\TT}\vdash_i \phi.$$
\end{theorem}
\begin{proof}
The proof of this theorem is essentially the same as the proof of the Theorem \ref{t8} by using the Theorem \ref{t12}. The only part that needs some extra work is the fact that $\C_{\M,\TT}\vdash_i \diag(\M)$ and moreover if $\C_{\M,\TT}\vdash_i p$ for atomic $p$, then $p\in\diag(\M)$. 

Let $p\in\diag(\M)$. We know by Theorem \ref{pr} $\M\models p\to\pr_\TT(\ulcorner p\urcorner)$. This implies $\M\models\pr_\TT(\ulcorner p\urcorner)$. So $p\in\ext(\M,\TT)$ which implies $\C_{\M,\TT}\vdash_i p$.

Now if we have $\C_{\M,\TT}\vdash_i p$ for some atomic $\L(\M)$ sentence $p$, then by Lemma \ref{l12} $\M\models\pr_\TT(\ulcorner p\urcorner)$. Note that $\M\models \con(\TT)$, so in presence of $\PRA$, $\M\models \Pi_1\text{-}\rfn(\TT)$ which $\Pi_1\text{-}\rfn(\TT)$ is the following sentence:
$$\all x(x\in\Pi_1\land \pr_\TT(x)\to\tr(x))$$
where $\tr$ is a natural $\Pi_1$ formula which works as the truth predicate for $\Pi_1$ sentence. Substituting $\ulcorner p\urcorner$ for $x$ in $\Pi_1\text{-}\rfn(\TT)$, we get $\M\models \tr(\ulcorner p\urcorner)$, hence $\M\models p$ which means $p\in\diag(\M)$.
\end{proof}

As we already see, using the first construction, we provide a Kripke model of $\HA+\ect$ which is not locally $\I\Delta_1$. A natural conjecture would be that the existence of such a Kripke model was possible because the base theory has a very powerful non-classical schema $\ect$. As an application of Theorem \ref{t14} we will show this is not the case. Let $\halt(x)$ be a $\Sigma_1$ formula that is a natural formalization of the statement "The Turing machine with code $x$ halts on input $x$". Let $\theta$ be an instance of $\ect$ in Definition \ref{d6} such that $\phi(x):=\top$ and $\psi(x,y):=(y=0\land \halt(x))\lor (y\neq 0\land \neg\halt(x))$. We also need the definition of Markov's principle.
\begin{definition}
Markov's principle is the following schema:
$$\mp:=\all \vec{y}(\all x(\phi(x,\vec{y})\lor\neg\phi(x,\vec{y}))\land \neg\neg\ex x\phi(x,\vec{y})\to\ex x\phi(x,\vec{y})) .$$
\end{definition}

\begin{lemma}\label{l17}
The following statements are true:
\begin{enumerate}
    \item $\HA+\neg\theta+\mp$ is consistent.
    \item $\HA+\neg\theta+\mp$ has the existence and disjunction properties.
\end{enumerate}
\end{lemma}
\begin{proof}
$ $
\begin{enumerate}
    \item It is easy to see that $\PA\vdash_c\neg\theta$ and also $\PA\vdash_c\mp$. So $\HA+\neg\theta+\MP$ is a sub-theory of $\PA$ and it is consistent.
    \item We will prove the existence property of $\HA+\neg\theta+\mp$ here. The disjunction property will follow from it like before. This part is a standard application of Kripke models (see \cite{sm}). Let $\ex x\psi(x)$ be an $\L$ sentence such that $\HA+\neg\theta+\mp\vdash_i \ex x\psi(x)$, but for every natural number $n$, $\HA+\neg\theta+\mp\not\vdash_i \psi(\bar{n})$. It is well-know that $\K(\mathbb{N},\HA+\neg\theta+\mp)$ is a well-defined Kripke model and moreover $\K(\mathbb{N},\HA+\neg\theta+\mp)\Vdash \HA$ (see Theorem 5.2.4 in \cite{sm}). Moreover we can assume that $\K_\bot$ (Note that $\bot\in \S(\mathbb{N},\HA+\neg\theta+\mp)$) is a Kripke model with just one node with the classical structure $\mathbb{N}$. Note that $r\nVdash \theta$, because otherwise by the monotonicity of forcing relation for every $\phi\in \S(\mathbb{N},\HA+\neg\theta+\mp)$, $\K_\phi\Vdash \theta$ which is not true. Moreover for every node $k\neq r$, $k\Vdash\neg\theta$, so with the last argument $r\Vdash \neg\theta$ which implies $\K(\mathbb{N},\HA+\neg\theta+\mp)\Vdash \neg\theta$. Note that $\mp$ is forced in every node $k\neq r$. So we only need to show that $r\Vdash \mp$. For this matter suppose $r\Vdash \all x(\phi(x,\vec{\bar{a}})\lor\neg\phi(x,\vec{\bar{a}}))\land \neg\neg\ex x\phi(x,\vec{\bar{a}})$ where $\vec{a}\in\mathbb{N}$. If for every $n\in\mathbb{N}$, $r\nVdash \phi(\bar{n},\vec{\bar{a}})$, then because of decidability of $\phi(x,\vec{\bar{a}})$ in the point of view of $r$, for every $n\in\mathbb{N}$, $r\Vdash \neg\phi(\bar{n},\vec{\bar{a}})$. This implies $\K_\bot\Vdash \all x\neg\phi(x,\vec{\bar{a}})$. But this leads to a contradiction because $\K_\bot\Vdash \neg\neg\ex x\neg\phi(x,\vec{\bar{a}})$. This means that there exists a natural number $n$ such that $r\Vdash \phi(\bar{n},\vec{\bar{a}})$.
    
    By the above arguments, we have $$\K(\mathbb{N},\HA+\neg\theta+\mp)\Vdash \HA+\neg\theta+\mp.$$ So $\K(\mathbb{N},\HA+\neg\theta+\mp)\Vdash \ex x\psi(x)$. This implies that there exists a natural number $n$ such that $r\Vdash \psi(\bar{n})$. But this leads to a contradiction because we know $\K_{\psi(\bar{n})}\nVdash \psi(\bar{n})$. This implies that our assumption was false and there exists a natural number $n$ such that $\HA+\neg\theta+\mp\vdash_i \psi(\bar{n})$.
\end{enumerate}
\end{proof}

The following corollary is the first application of Theorem \ref{t14}.
\begin{corollary}\label{finF}
There exists a rooted Kripke model of $\HA+\neg\theta+\mp$ which is not locally $\I\Delta_1$.
\end{corollary}
\begin{proof}
By Theorem \ref{t9} there exists a model $\M\models \th_{\Pi_2}(\mathbb{N})$ such that $\M\not\models {\bf B}\Sigma_1$ and hence by Theorem \ref{eq} $\M\not\models\I\Delta_1$. Note that by Lemma \ref{l17} $\HA+\neg\theta+\mp$ is consistent and has the existence property. This implies that $\ep(\HA+\neg\theta+\mp)$ and $\con(\HA+\neg\theta+\mp)$ are true in $\mathbb{N}$. Note that these sentences are $\Pi_2$, so they are also true in $\M$. This implies that $\M$ satisfies the conditions needed in the Theorem \ref{t14}, hence $$\K(\M,\C_{\M,\HA+\neg\theta+\mp})\Vdash \HA+\neg\theta+\mp$$ and also it is not locally $\I\Delta_1$.
\end{proof}

It is worth mentioning that $\HA+\neg\theta+\mp$ does not prove anything contradictory with $\PA$ and in some sense, it is close to $\PA$, but still, we were able to construct a Kripke model of it which is not locally $\I\Delta_1$.
The following corollary is the second application of Theorem \ref{t14}.
\begin{corollary}\label{forceHA}
There exists a rooted Kripke model $\K\Vdash\HA$ which is not locally $\I\Delta_1$, but for every $\L$ sentence $\phi$,
$$\K\Vdash \phi \Leftrightarrow \HA\vdash_i\phi.$$
\end{corollary}
\begin{proof}
Define $$\U=\{\neg\pr_\HA(\ulcorner\phi\urcorner)\:|\:\HA\not\vdash_i\phi,\phi\text{ is a sentence}\}.$$
Let $\TT:=\PRA+\ep(\HA)+\U$. It is easy to see that $\TT$ is a $\Pi_2$ axiomatized theory and moreover $\mathbb{N}\models\TT$. By Theorem \ref{t9} there exists a model $\M\models\th_{\Pi_2}(\mathbb{N})$ such that $\M\not\models\I\Delta_1$. By the facts that $\mathbb{N}\models\TT$ and also $\TT$ is a $\Pi_2$ axiomatized theory, we get $\M\models\TT$. So by these explanations, $\M$ has the required property that is needed in Theorem \ref{t14}, hence $\K(\M,\C_{\M,\HA})\Vdash \HA$. This means that for every $\L$ sentence $\phi$, if $\HA\vdash_i\phi$, then $\K(\M,\C_{\M,\HA})\Vdash\phi$. 

For the opposite direction, let $\phi$ be an $\L$ sentence such that $\K(\M,\C_{\M,\HA})\Vdash \phi$. Then by Theorem \ref{t14} $\C_{\M,\HA}\vdash_i\phi$. So by Lemma \ref{l12} $\M\models\pr_\HA(\ulcorner\phi\urcorner)$. If $\HA\not\vdash_i\phi$, then $\neg\pr_\HA(\ulcorner\phi\urcorner)\in\U$, hence $\TT\vdash_c\neg\pr_\HA(\ulcorner\phi\urcorner)$ which implies $\M\models\neg\pr_\HA(\ulcorner\phi\urcorner)$, but this leads to a contradiction, hence $\HA\vdash_i\phi$.
\end{proof}

As we already mentioned in the Introduction, we can get more Kripke models for $\HA+\ect$ from the first construction than by the second construction. We will show this fact in the rest of this subsection. For this matter, we need the following theorem.
\begin{theorem}\label{emil}
For any constant $k$, there is no consistent $\Pi_k$-axiomatized theory $\TT$ such that $\TT\vdash_c \PA$.
\end{theorem}
\begin{proof}
See \cite{jer}.
\end{proof}

\begin{theorem}
The following statements are true:
\begin{enumerate}
    \item For every $\L$ structure $\M$, if $\K(\M,\C_{\M,\HA+\ect})\Vdash \HA+\ect$, then $\K(\M,\T)\Vdash \HA+\ect$.
    \item There exists an $\L$ structure $\M$ such that $\K(\M,\T)\Vdash \HA+\ect$, but $\K(\M,\C_{\M,\HA+\ect})\nVdash \HA$.
\end{enumerate}
\end{theorem}
\begin{proof}
$ $

\begin{enumerate}
    \item Suppose $\K(\M,\C_{\M,\HA+\ect})\Vdash \HA+\ect$. Let $\phi:=\all \vec{x}\ex\vec{y}\psi(\vec{x},\vec{y})$ be a $\Pi_2$ sentence such that $\PA\vdash_c \phi$. Then by $\Pi_2$ conservativity of $\HA$ over $\PA$, we have $\HA\vdash_i \phi$, hence $\K(\M,\C_{\M,\HA+\ect})\Vdash \phi$. This implies $r\Vdash \all \vec{x}\ex\vec{y}\psi(\vec{x},\vec{y})$. So for every $\vec{a}\in\M$:
    \begin{enumerate}
        \item $\Rightarrow r\Vdash \ex\vec{y}\psi(\un{\vec{a}},\vec{y})$,
        \item $\Rightarrow\text{ there exist } \vec{b}\in\M\text{ such that }r\Vdash \psi(\un{\vec{a}},\un{\vec{b}})$,
        \item $\Rightarrow \M\models\psi(\un{\vec{a}},\un{\vec{b}})$.
    \end{enumerate}
    Hence $\M\models\phi$. This implies that $\M\models \th_{\Pi_2}(\PA)$, so by Theorem \ref{t8} $\K(\M,\T)\Vdash \HA+\ect$.
    \item By G\"odel's second incompleteness theorem, $\PA+\neg\con(\PA)$ is consistent. So this implies that $\th_{\Pi_2}(\PA)+\neg\con(\HA)$ is also consistent. $\th_{\Pi_2}(\PA)+\neg\con(\HA)$ is a $\Pi_2$-axiomatized theory, hence by Theorem \ref{emil} there exists a model $\M\models\th_{\Pi_2}(\PA)+\neg\con(\HA)$ such that $\M\not\models\PA$. Note that by Theorem \ref{t8} $\K(\M,\T)\Vdash \HA+\ect$. On the other hand $\M\models\neg\con(\HA+\ect)$, so $\bot\in\ext(\M,\HA+\ect)$. This implies $\C_{\M,\HA+\ect}\vdash_i\bot$. Hence $\S(\M,\C_{\M,\HA+\ect})=\varnothing$. This means that $\K(\M,\C_{\M,\HA+\ect})$ has only one node $r$ such that $\M_r=\M$. Note that $\M\not\models\PA$, so $r\not\Vdash\HA$ and this completes the proof.
\end{enumerate}
\end{proof}
\section{On binary Kripke models for intuitionistic first-order logic}
In this section, we will prove that every countable rooted Kripke model $\K$ (there exists a node $k$ in $\K$ such that for every $k$ in $\K$, $k\leq k'$) can be transformed to a Kripke model $\K'$ with the infinite full binary tree as Kripke frame such that $\K$ and $\K'$ force the same sentences. This was known for the case of finite Kripke models of intuitionistic propositional logic (see Theorem 2.21 and Corollary 2.22 of \cite{modal}), but to best of our knowledge it was not mentioned for the case of Kripke models of intuitionistic first-order logic in the literature. The transformation for Kripke models of intuitionistic first-order logic can be done in the same way that was done for the case of finite Kripke models of intuitionistic propositional logic, but for the sake of completeness we will state the theorem and prove it in this section.

Let $\Gamma=\{0,1\}$ and $\Gamma^*$ be the set of all finite binary strings (including empty string $\lambda$). For every $x,y\in\Gamma^*$, $x\preceq y$ iff $x$ is a prefix of $y$. 
\begin{lemma}\label{bin}
Let $\K=(K,\leq,\M)$ be a countable rooted Kripke model in a language $\sigma$. Then there is an onto function $f:\Gamma^*\to K$, such that:
\begin{enumerate}
    \item $\K'=(\Gamma^*,\preceq,\M')$ is a Kripke model where $\M'$ is defined as $\M'_x=\M_{f(x)}$ for every $x\in\Gamma^*$,
    \item for every $k\in K$, for every $\sigma({\M_k})$ sentence $\phi$, and for every $x\in\Gamma^*$ such that $f(x)=k$, $x\Vdash \phi$ iff $k\Vdash \phi$.
\end{enumerate}
\end{lemma}
\begin{proof}
Without loss of generality, we can assume $(K,\leq)$ is a tree (see Theorem 6.8 in the second chapter of \cite{tv}) with the root $r$. 
Also, we can assume that for every $k\in K$, there is a $k'\in K$ different from $k$ such that $k\leq k'$. This is true because for every $k\in K$ that does not have relation with any other nodes, we can put an infinite countable path above $k$ such that the classical structure of every node in this path is $\M_k$. This transformation does not change the sentences that were forced in the original model. 
For every $k\in K$, define neighbor of $k$ as $$\N_k=\{k'\in K| k\leq k'\land k\neq k' \land \all k''\in K(k\leq k'' \land k''\leq k'\to k=k'' \lor k'=k'')\}.$$ 
For every $k\in K$, fix an onto function $g_k:\mathbb{N}\to \N_k$ such that for every $k'\in \N_k$, $\{ n\in\mathbb{N}|g_k(n)=k'\}$ is infinite. Now we define $f$ inductively with a sequence of partial function $f_0\subset f_1 \subset ...$ and then we put $f=\bigcup_{n\in\mathbb{N}}f_n$. Put $f_0(\lambda)=r$. For a function $h$, let $\D(h)$ be domain of $h$. Let $${\cal A}_n=\{x\in\Gamma^*|x\in\D(f_n),x0\not\in\D(f_n),x1\not\in\D(f_n)\}.$$ Now $f_{n+1}$ is defined inductively from $f_n$ as follows:
$$f_{n+1}(x)=\begin{cases}f_n(x) & x\in{\sf Dom}(f_n)\\f_n(y) & x=y0^m, \text{ for some }y\in{\cal A}_n, m\in\mathbb{N}\\
g_{f_{n}(y)}{(m)} & x=y0^m1, \text{ for some }y\in{\cal A}_n, m\in\mathbb{N}.
\end{cases}$$
It is easy to see that $\D(f)=\Gamma^*$.
\begin{claim}\label{cl}
For every $k\in K$, for every $x\in\Gamma^*$ if $f(x)=k$, then $$\{k'\in K| k\leq k'\}=\{f(y)\in K| y\in\Gamma^*, x\preceq y\}.$$
\end{claim}
This claim is easy to prove considering the definition of $f$ and the fact that $g_k$ functions enumerate neighbors infinitely many times.

Using this claim, we can finish the proof. The proof goes by induction on the complexity of $\phi$. We will only mention a nontrivial case in the induction steps. All other cases can be treated similarly. Let $\phi:= \psi\to \eta$ and $k\Vdash \psi\to\eta$. Let $x\in \Gamma^*$ be such that $f(x)=k$. Suppose for some $y \succeq x$, we know $y\Vdash \psi$. So by the induction hypothesis, $f(y)\Vdash \psi$ and by Claim \ref{cl}, we know $f(y)\geq k$, hence $f(y)\Vdash \eta$, therefore by the induction hypothesis we get $y\Vdash \eta$, so $x\Vdash \phi$.
\end{proof}
\begin{corollary}\label{bin2}
There exists a Kripke model of $\HA$ with $(\Gamma^*,\preceq)$ as the Kripke frame that is not locally $\I\Delta_1$.
\end{corollary}
\begin{proof}
Let $\K$ be a rooted Kripke model with the root $r$ in a language $\sigma$. Let $\cal U$ be a countable set of sentences of $\sigma$. It is easy to see that $\K$ can be represented by a suitable two-sorted classical structure $\MM_\K$ such that:
\begin{enumerate}
    \item For every $\phi\in{\cal U}$, "$r\Vdash \phi$" is first-order definable in $\MM_\K$ by the sentence $\phi_F$.
    \item For every $\phi\in{\cal U}$, "$\M_r\models \phi$" is first-order definable in $\MM_\K$ by the sentence $\phi_M$.
\end{enumerate}
By applying the downward L\"owenheim–Skolem theorem on $\MM_\K$ we get a countable substructure of $\MM_\K$ like $\MM'_\K$ such that:
\begin{enumerate}
    \item $\MM'_\K$ is a representation of a countable rooted Kripke model in the language $\sigma$.
    \item For every $\phi\in{\cal U}$, $\MM_\K\models \psi$ iff $\MM'_\K\models \psi$, for $\psi\in\{\phi_F,\phi_M\}$.
\end{enumerate}
Let $\KM$ be the rooted Kripke model from Corollary \ref{fin}. Let ${\cal U}=\HA\cup\{\varphi\}$ where $\varphi$ is an instance of $\Delta_1$ induction that fails in the classical structure of the root of $\KM$. Following the same argument on $\KM$ and $\cal U$, we get a countable rooted Kripke model $\K'$ of $\HA$ that is not locally $\I\Delta_1$. Hence applying Lemma \ref{bin} on $\K'$ finishes the proof.
\end{proof}
\section{Concluding remarks and open problems}
Problem \ref{p1} can be asked about other theories than $\HA$. One can ask the same question about arithmetic over sub-intuitionistic logic too. One of these logics is Visser's Basic logic, and its extension Extended Basic logic. The model theory of arithmetic over these logics were investigated in \cite{ru,ah2,aks2}. From the point of view of Problem \ref{p1}, it is proved in \cite{ah} that every irreflexive node in a Kripke model of $\sf BA$ (Basic Arithmetic) is locally $\I\ex^+_1$. So In general, every irreflexive node in a Kripke model of the natural extension of $\sf BA$ such as $\sf EBA$ (Extended Basic Arithmetic) is locally $\I\Sigma_1$ (see Corollary 3.33 in \cite{aks2}). Also it is proved in \cite{aks2} that every Kripke model of $\sf EBA$ is locally $\th_{\Pi_2}(\I\Sigma_1)+\th_{\Pi_1}(\PA)$. Note that every Kripke model of $\HA$ is also a Kripke model of $\sf BA$ and $\sf EBA$. So Corollary \ref{fin} applies to these theories too, and this solves Problem \ref{p1} for these theories. Furthermore, this shows that the known positive results are the best we can get for $\sf BA$ and $\sf EBA$.

Focusing on the proof of Theorem \ref{dis}, we essentially use $\ect$ for proving the existence and disjunction properties of $\T$. We do not know whether $\ect$ is essential for such a model construction, so we have the following question:
\begin{problem}
        Does theory $\HA+\diag(\M)$ has the existence property for every $\M\models \th_{\Pi_2}(\PA)$?
\end{problem}
 An important problem which we could not answer is the following:
\begin{problem}
         Is there any Kripke model $\K\Vdash \HA$ such that for every node $k$ in $\K$, $\M_k\not\models \PA$? 
\end{problem}

Another unsolved question in the direction of completeness with respect to locally $\PA$ Kripke models is the following:
\begin{problem}
         Does $\HA$ have completeness with respect to its class of locally $\PA$ Kripke models?
\end{problem}
By the result of \cite{bu}, for every sentence $\phi$ such that $\HA\nvdash_i \phi$, there exists a locally $\PA$ Kripke model $\K$ such that $\K\not\Vdash \phi$, but this result does not say anything about whether $\K$ is a Kripke model of $\HA$ or not.

We call a rooted tree Kripke frame $(K,\leq)$, a $\PA$-frame iff for every Kripke model $\K\Vdash \HA$ with frame $(K,\leq)$, $\K$ is locally $\PA$. Let ${\mathcal{F}}_\PA$ be the set of all $\PA$-frames. We know that semi narrow rooted tree Kripke frames are in ${\cal F}_\PA$. On the other hand, by Corollary \ref{cl} infinite full binary tree is not in ${\cal F}_\PA$. So we have the following question:
\begin{problem}
         Is there a nice characterization of ${\cal F}_\PA$?
\end{problem}
\subsection*{Acknowledgment}
    We are indebted to Mohammad Ardeshir and truly grateful to him for his careful guidance, invaluable academic teachings, and many invaluable discussions that we have had during the studies in the Department of Mathematical Sciences of the Sharif University of Technology which had a clear impact on our academic life. We also thank him for fruitful discussions about this work. We are grateful to Mohsen Shahriari for fruitful discussions about this work and also reading the draft of this paper and pointing out a gap in the proof of Corollary \ref{bin2}. We are grateful to Pavel Pudl\'ak for fruitful discussions about this work and also reading the draft of the paper and pointing out some English errors in it and also comments which led to a better presentation of the work. We thank Emil Je\v r\'abek for a discussion about this work and also reading the draft of the paper and his comments on it. We also thank Sam Buss, Fedor Pakhomov, and Albert Visser for discussions about this work and answering our questions. The first model construction was done when the author was at the Sharif University of Technology. The second construction was proved while the author was in the Institute of Mathematics of the Czech Academy of Sciences. This research was partially supported by the project EPAC, funded by the Grant Agency of the Czech Republic under
the grant agreement no. 19-27871X.

\end{document}